\documentclass[10pt,
twoside]{amsart}
\usepackage{lmodern} 
\usepackage{soul, wrapfig}
\usepackage{graphicx}
\usepackage{color}
\usepackage{amsmath}
\usepackage{amssymb}
\usepackage{amsfonts}
\usepackage{latexsym, amsthm, mathrsfs}
\usepackage{hyperref, breakurl}
\usepackage[left]{showlabels}
\usepackage[OT1]{fontenc}
\usepackage{fancybox}
\usepackage{amscd}
\usepackage{enumitem}
\usepackage{fontenc}
\usepackage{amsthm}
\usepackage{graphicx}
\usepackage[english]{babel}
\usepackage{tikz-qtree}
\usepackage{rotating}
\usepackage{stmaryrd}

\usepackage[]{graphicx}
\usepackage{fontenc}
\usepackage{amsmath}
\usepackage{amsfonts}
\usepackage{amssymb}
\usepackage{amsthm}
\usepackage{graphicx}
\usepackage[english]{babel}
\usepackage{pifont}

\numberwithin{equation}{section}
\newtheorem{theorem}{Theorem}[section]
\newtheorem{lemma}[theorem]{Lemma}

\theoremstyle{plain} 

\newtheorem{proposition}[theorem]{Proposition}

\newtheorem*{maintheorem*}{Main Theorem}
\newtheorem*{conjecture*}{Conjecture}
\newtheorem*{theorem*}{Theorem}
\newtheorem*{proposition*}{Proposition}
\newtheorem*{corollary*}{Corollary}

\theoremstyle{definition} 
\newtheorem{definition}[theorem]{Definition}

\theoremstyle{remark}  
\newtheorem{remark}[theorem]{Remark}
\newtheorem{example}[theorem]{Example}
\newtheorem*{remarks*}{Remarks}
\newtheorem*{remark*}{Remark}
\newtheorem*{claim*}{Claim}

\newcommand{\cmark}{\ding{51}}

\newcommand{\cohen}{\poset{C}}
\newcommand{\conc}{\smallfrown}
\newcommand{\cov}{\textsf{cov}}
\newcommand{\club}{\mathrm{Club}}
\newcommand{\clubmiller}{\miller^\club_\kappa}

\newcommand{\dom}{\operatorname{dom}}

\newcommand{\ideal}{\mathcal}

\newcommand{\ifif}{\Leftrightarrow}

\newcommand{\force}{\Vdash}

\newcommand{\mathias}{\poset{T}}
\newcommand{\meager}{\ideal{M}}

\newcommand{\miller}{\poset{M}}

\newcommand{\poset}{\mathbb}

\newcommand{\random}{\poset{B}}

\newcommand{\restric}{{\upharpoonright}}

\newcommand{\silver}{\poset{V}}

\newcommand{\splitting}{\mathsf{Split}}

\newcommand{\stem}{\textsf{stem}}
\newcommand{\successor}{\textsf{succ}}
\newcommand{\such}{\; : \;}

\newcommand{\xmark}{\ding{55}}

\begin{document}

\title[More on trees and Cohen reals\today]{More on trees and Cohen reals}
\author{Giorgio Laguzzi, Brendan Stuber-Rousselle}

\begin{abstract}
In this paper we analyse some questions concerning trees on $\kappa$, both for the  countable and the uncountable case, and the connections with Cohen reals. In particular, we provide a proof for one of the implications left open  in \cite[Question 5.2]{FKK16} about the diagram for regularity properties. 
\end{abstract}

\maketitle

\section{Introduction}

Throughout the paper we deal with trees on $\eta^{<\kappa}$, with $\kappa \geq \omega$ being any regular cardinal and $\eta\geq 2$ or if $\eta$ is infinite then $\eta$ regular too.

A tree-forcing $\poset{P}$ is a poset whose conditions are perfect trees $p \subseteq \eta^{<\kappa}$ with the property that for every $p \in \poset{P}$ and every $t \in p$ one has $p \restric t := \{ t' \in p: t' \subseteq t \vee t \subseteq t' \} \in \poset{P}$; the ordering is $q \leq p \Leftrightarrow q \subseteq p$. In case $\kappa=\omega$ and $\eta \in \{ 2, \omega \}$ some of the most popular tree-forcings are for instance: the Hechler forcing $\poset{D}$ (\cite[Def. 3.1.9, p.104]{bart}), eventually different forcing $\poset{E}$ (\cite[Def. 7.4.8, p.366]{bart}), Sacks forcing (see \cite[p.3]{Bre95}), Silver forcing $\poset{V}$ (see \cite[p.4]{Bre95}), Miller forcing $\poset{M}$ (see \cite[p.3]{Bre95}), Laver forcing (see \cite[p.3]{Bre95}), Mathias forcing $\poset{R}$ (see \cite[p.4]{Bre95}), random forcing $\random$ (see \cite[p. 99]{bart}). The relation between tree-forcings and Cohen reals has been rather extensively developed in the literature. The reason to study such connections for different types of tree-forcing notions was mainly to ``separate" different kinds of cardinal characteristics, in particular from $\cov(\meager)$. 
We can associate a tree-forcing $\poset{P}$ in a standard way with a notion of $\poset{P}$-nowhere dense sets, $\poset{P}$-meager sets and $\poset{P}$-measurable sets. 
\begin{definition}\label{meager}
\item Given $\poset{P}$ a tree-forcing notion and $X\subseteq \eta^\kappa$ a set of $\kappa$-reals, we say that: 
\begin{itemize}
\item $X$ is \emph{$\poset{P}$-nowhere dense} if 
\[
\forall p \in \poset{P} \exists q \leq p ([q] \cap X=\emptyset),
\]
and we put $\ideal{N}_\poset{P}:= \{ X: X \text{ is $\poset{P}$-nowhere dense} \}$.

\item $X$ is \emph{$\poset{P}$-meager} if there are $A_i \in \ideal{N}_\poset{P}$ such that $X \subseteq \bigcup_{i \in \kappa} A_i$, and we put 
$\ideal{I}_\mathbb{P}=\{ X: X \text{ is $\poset{P}$-meager}  \}$. 

\item $X$ is \emph{$\poset{P}$-measurable} if 
\[
\forall p \in \poset{P} \exists q \leq p ([q] \cap X \in \ideal{I}_\poset{P} \vee [q] \setminus X \in \ideal{I}_\poset{P}).
\]
\item A family $\Gamma$ of subsets of $\kappa$-reals is called \emph{well-sorted} if it is closed under continuous pre-images. We abbreviate the sentence ``every set in $\Gamma$ is $\mathbb{P}$-measurable" by $\Gamma(\mathbb{P})$. 
\end{itemize}
\end{definition}

For example when $\mathbb{P}$ is the Cohen forcing $\cohen$, then $\mathbb{C}$-meagerness coincides with topological meagerness and $\mathbb{C}$-measurability coincides with the Baire Property. When $\mathbb{P}$ is the Random forcing $\mathbb{B}$, then $\mathbb{B}$-meagerness coincides with Lebesgue measure zero and $\mathbb{B}$-measurability coincides with Lebesgue measurability.\\
The presence of Cohen reals added by a tree-forcing $\poset{P}$ has an impact both on the structure of $\ideal{I}_\poset{P}$ and on the corresponding notion of $\poset{P}$-measurability, as specified in the tables introduced below. 
More specifically, if $\poset{P}$ adds a Cohen real then the way of coding the $\poset{P}$-generic into a Cohen real often induces a construction providing $\Gamma(\poset{P}) \Rightarrow \Gamma(\cohen)$ (e.g., see \cite[Theorem 3.1]{BL99} where such a connection is shown in case of $\poset{P}=\poset{D}$). Moreover the presence of a coded Cohen real often implies that $\ideal{N}_\poset{P}$ and $\ideal{I}_\poset{P}$ do not coincide. For instance, this holds for the Hechler forcing $\poset{D}$ and for the eventually different forcing $\poset{E}$. Both these forcings are ccc, and indeed $\sigma$-centered. So, a natural question that arises is whether one can find a non-ccc tree-forcing notion $\poset{P}$ for which $\Gamma(\poset{P}) \Rightarrow \Gamma(\cohen)$ and $\ideal{I}_\poset{P} \neq \ideal{N}_\poset{P}$. In this paper we give a positive answer, by defining and analysing a variant of Mathias forcing in the space $3^\omega$ instead of $2^{\omega}$. 

As a more general question, for a tree-forcing $\poset{P}$, one can consider the four properties mentioned so far, namely: 1) $\poset{P}$ adds Cohen reals; 2) $\Gamma(\poset{P})\Rightarrow \Gamma(\cohen)$; 3) $\ideal{I}_\poset{P} \neq \ideal{N}_\poset{P}$; 4) $\poset{P}$ is ccc.
So for instance, if we consider the most popular tree-forcings we get the following table, where $\poset{T}$ stands for the variant of Mathias forcing defined in Section \ref{section2}, and $\miller^{\text{full}}$ is the variant of Miller forcing where we require that every splitting node splits into the whole $\omega$. The results in Table 1 without an explicit reference are deemed as folklore. 

\vspace{3mm}
\begin{center}
\begin{tabular}{|c|c|c|c|c|}\hline
\multicolumn{5}{|c|}{\textbf{Table 1}} \\ \hline	
 & Adding Cohen  & $\mathcal{I}_\mathbb{P}\neq \mathcal{N}_\mathbb{P}$ & $\Gamma(\mathbb{P})\Rightarrow \Gamma(\mathbb{C})$ & c.c.c\\ \hline
$\mathbb{D}$, $\mathbb{E}$ & \cmark	 & \cmark & \cmark (\cite[Theorem 3.1]{BL99} ) & \cmark  \\ \hline
$\random$ & \xmark	 & \xmark & \xmark (\cite{Sh85}) & \cmark \\ \hline
$\mathbb{V}$, $\miller$, $\poset{R}$ & \xmark	 & \xmark & \xmark &  \xmark \\ \hline
$\poset{T}$ & \cmark (Lemma \ref{t adds cohen reals})	 & \cmark (Lemma \ref{t^0 T^0}) & \cmark (Proposition \ref{prop-mathias-baire}) & \xmark  \\ \hline
$\poset{M}^{\text{full}}$ & \cmark	 & \xmark & \cmark (\cite[Theorem 3.4]{KL2018}) & \xmark  \\ \hline
\end{tabular}
\end{center}
\text{ }\\\\

Note that the table above refers to the tree-forcings in the $\omega$-case, and so defined on spaces like $2^\omega$, $\omega^\omega$ or $[\omega]^\omega$. 

For $\kappa > \omega$ we could consider the same table, but then the situation changes and we can get several different developments. We always assume $\kappa^{<\kappa}=\kappa$.
\begin{enumerate}
\item For $\poset{D}_\kappa$ (and similarly for $\poset{E}_\kappa$), the constructions done for the $\omega$-case (e.g., the proof of \cite[Theorem 3.1]{BL99}) easily generalises; 
\item for the $\kappa$-Silver forcing, the situation seems to depend on whether $\kappa$ is inaccessible or not; but it is rather independent of whether we consider club splitting or other version of $<\kappa$-closure;
\item for $\kappa$-Mathias forcing, the situation is drastically different from the $\omega$-case, as we can prove a strict connection with the Baire property and Cohen reals;
\end{enumerate}

The table for $\kappa$ uncountable then appears as follows, where $\kappa$ denotes any cardinal, $\lambda$ any inaccessible cardinal and $\gamma$ any not inaccessible cardinal:

\vspace{3mm}
\begin{center}
\begin{tabular}{|c|c|c|c|c|}\hline
\multicolumn{5}{|c|}{\textbf{Table 2}} \\ \hline	
 & Adding Cohen  & $\mathcal{I}_\mathbb{P}\neq \mathcal{N}_\mathbb{P}$ & $\Gamma(\mathbb{P})\Rightarrow \Gamma(\mathbb{C})$ & $\kappa^+$-c.c\\ \hline
$\mathbb{D}_\kappa$, $\mathbb{E}_\kappa$ & \cmark (Definition 48 \cite{BBFM})	 & \cmark & \cmark (Reamrk \ref{reamrk-kappa-hechler}) & \cmark  \\ \hline
$\poset{M}^\club_\kappa$ & \cmark (Proposition 77 \cite{BBFM})	 & \xmark (Lemma 3.8. \cite{FKK16}) & \cmark & \xmark  \\ \hline
$\silver^\club_\lambda$ & \xmark	 & \xmark & \xmark (Theorem 4.11. \cite{FKK16})&  \xmark \\ \hline
$\silver^\club_\gamma$ & ? & ? & ? &  \xmark \\ \hline
$\poset{R}^\club_\kappa$ & \cmark (Remark 30 \cite{Lag16})	 & \cmark(Lemma 4.1. \cite{FKK16}) & \cmark &  \cmark  \\ \hline
$\poset{R}_\kappa$ & \cmark (Remark 30 \cite{Lag16})	 & \cmark (Lemma \ref{lemma:kappa-mathias-ideal}) & \cmark (Proposition 31 \cite{Lag16}) &  \xmark \\ \hline
\end{tabular}
\end{center}
\text{ }\\\\

\paragraph{Basic notions and definitions}

The elements in $\eta^\kappa$ are called $\kappa$-reals or $\kappa$-sequences, where $\eta$ is also a regular cardinal, usually $\eta=2$ or $\eta=\kappa$.  Given $s,t \in \eta^{<\kappa}$ we write $s \perp t$ iff neither $s \subseteq t$ nor $t \subseteq s$ (and we say $s$ and $t$ are incompatible).
The following notations are also used.
\begin{itemize}

\item A \emph{tree} $p \subseteq \eta^{<\kappa}$ is a subset closed under initial segments and its elements are called \emph{nodes}. We consider $< \kappa$-closed trees $p$, i.e., for every $\subseteq$-increasing sequence of length $<\kappa$ of nodes in $p$, the supremum (i.e., union) of these nodes is still in $p$. Moreover, we abuse of notation denoting by $|t|$ the ordinal $\dom(t)$. 
\item We say that a $<\kappa$-closed tree $p$ is \emph{perfect} iff for every $s \in p$ there exists $t \supseteq s$ and $\alpha, \beta \in \eta$, $\alpha \neq \beta$, such that $t^\conc \alpha \in p$ and $t^\conc \beta \in p$; we call such $t$ a \emph{splitting node} (or \emph{splitnode}) and set $\splitting(p):=\{ t \in p: t \text{ is splitting} \}$.
\item 
We say that a splitnode $t \in p$ has \emph{order type} $\alpha$ (and we write $t \in \splitting_\alpha(p)$) iff $\operatorname{ot}(\{ s \in p: s \subsetneq t \land s \in \splitting(p) \},\subsetneq)=\alpha$.
\item $\stem(p)$ is the longest node in $p$ which is compatible with every node in $p$; $p \restric t:= \{ s \in p: s \text{ is compatible with } t \}$.
 \item $[p]:= \{ x \in \eta^\kappa: \forall \alpha < \kappa (x \restric \alpha \in p) \}$ is called the \emph{set of branches} (or \emph{body}) of $p$. 
\item $\successor(t,p):= \{ \alpha \in \eta: t^\conc \alpha \in p \}$, for $t \in p$.
\item A poset $\poset{P}$ is called \emph{tree-forcing} if its conditions are perfect trees and for every $p \in \poset{P}$, and every $t \in p$, one has $p \restric t \in \poset{P}$ too.
\end{itemize}

\begin{remark}
When comparing different notions of $\poset{P}$-measurablity, i.e., investigating the relationship between $\Gamma(\poset{P})$ and $\Gamma(\poset{Q})$ for different tree-forcings $\poset{P}$ and $\poset{Q}$, we often refer to different topological spaces. As Brendle pointed out explicitly in \cite{Bre95} the idea is to consider the analogue versions in the space of strictly increasing sequences $\omega^{\uparrow \omega}$ which can be seen to be almost isomorphic to the spaces we deal with (for the details see paragraph 1.2 in \cite{Bre95}). The only case that is not covered in \cite{Bre95} is $3^\omega$. In this paper we need to implement this case as well, as we are going to work with it in the coming section. Actually in trying to describe a suitable isomorphism, we need to consider a special subspace, in the same fashion as we do when we consider only the subspace of $2^\omega$ consisting of binary sequences that are not eventually 0. Analogously we consider $H:= \{ x \in 3^\omega: \exists^\infty n (x(n)=2) \}$ and we define the appropriate map $\varphi: H \rightarrow \omega^{\uparrow\omega}$ as follows: we fix the lexicographic enumeration $b: 2^{<\omega} \rightarrow \omega$. So $b(s)\leq b(t)$, whenever $s\subseteq t$ and in particular $b(\langle \rangle) = 0$. For every $x \in H$ let $\{ n_k: k \in \omega \}$ enumerate the set of all inputs $n$ such that $x(n)=2$. Then define $\sigma^x_0:=  \langle x(i): 0 \leq i < n_{0} \rangle$ and for every $j \in \omega$, $\sigma^x_{j+1}:= \langle x(i): n_j < i < n_{j+1} \rangle$. Finally put 
\[
\varphi(x):=  \langle b(\sigma^x_0), b(\sigma^x_0) + b(\sigma^x_1) + 1, b(\sigma^x_0) + b(\sigma^x_1) + b(\sigma^x_2) + 2, \dots \rangle = \langle \sum_{i\leq n} b(\sigma^x_i) + n: n \in \omega \rangle.
\]
One can easily check that $\varphi$ is an isomorphism.
\end{remark}

\section{A variant of Mathias forcing} \label{section2}
\begin{definition}\label{variant of mathias forcing}
We define $\mathias$ as the tree-forcing consisting of perfect trees $p \subseteq 3^{<\omega}$ with $A_p \subseteq \omega$ such that:
\begin{itemize}
\item for every $t \in p$ $(|t| \in A_p \ifif t \in \splitting(p))$, we refer to $A_p$ as the set of splitting levels of $p$;
\item if $t \in \splitting(p)$, then $t$ is fully splitting (i.e., for every $i \in 3$, $t^\conc i \in p$);
\item for every $s \supseteq \stem(p)$, if $s \notin \splitting(p)$ then $s^\conc 2 \notin p$;
\item for every $s,t \in p$, $|s|=|t|$, $s,t \notin \splitting(p)$, one has 
\[
\forall i \in 2 (s^\conc i \in p \ifif t^\conc i \in p).
\]
\end{itemize}
 \end{definition}
 
Intuitively, any condition $p\in \mathbb{T}$ is a perfect tree in $3^{<\omega}$ such that at any level $n\in \omega$ either $p$ uniformly splits, or uniformly takes the same value.

Note that $\mathbb{T}$ is not $c.c.c.$.
To show that let $E\subseteq \omega$ be the set of even numbers and $O= \omega \setminus E$. For each $a\subseteq O$ we define a condition $p_a\in \mathbb{T}$ in the following way: on even levels we uniformly split and on odd levels $n$ we uniformly choose the value $1$ whenever $n\in a$ and $0$ otherwise, so
$$p_a := \{ t\in 3^{<\omega} \such \forall n\in O \cap |t|\; ( (n \in a \rightarrow  t(n)=1)\wedge (n\not\in a \rightarrow t(n)=0 ))  \}.$$ 
We claim that $\{p_a \such a\subseteq O\}$ is an antichain. In fact, let  $a, b \subseteq O$ be two different subsets and fix $n\in O$ such that $n\in a\setminus b$ or $n\in b\setminus a$. W.l.o.g. assume $n\in a\setminus b$. Then each branch $x$ through $p_a$ must satisfy $x(n)=1$, whereas each branch $y$ through $b$ satisfies $y(n)=0$. Thus $[p_a]\cap [p_b]= \emptyset$ and in particular $p_a\perp p_b$.

\vspace{0.5cm}

Under a certain point of view $\mathias$ seems to behave like the original Mathias forcing $\poset{R}$. For instance, the following proof showing that $\mathias$ satisfies Axiom A follows the same  line as for $\poset{R}$. However, going more deeply one has to be careful, as even if $\mathias$ still satisfies quasi pure decision (Lemma \ref{quasi-pure}), it fails to satisfy pure decision (Lemma \ref{t adds cohen reals}). Thus, we examine these proofs in closer detail to better understand the main differences between $\mathias$ and $\poset{R}$.  

\begin{proposition}
$\mathias$ satisfies Axiom A.
\end{proposition}

\begin{proof}
We define the partial orderings $\langle\leq_n \such n\in \omega\rangle$ in the expected way: For $p,q\in \mathbb{T}$ we put $q\leq_n p$ if and only if $q\leq p$ and the two sets of splitting levels $A_q$ and $A_p$ coincide on the first $n+1$ elements. So, in particular $q\leq_0 p$ implies $\stem(q)=\stem(p)$. It is easy to check that fusion sequences exist. Let $p\in \mathias$, $k\in \omega$ and $D\subseteq \mathias$ a dense subset be given. We show that there is a stronger condition $q\leq_k p$ and a finite set $E\subseteq D$ pre dense below $q$. This proves that $\mathias$ satisfies Axiom A. Let $A_p=\{n_i \such i<\omega  \}$ be an increasing enumeration of the splitting levels of $p$. Observe that there are exactly $3^k$ nodes $t\in p$ of length $n_k$. Each of those nodes is splitting, so that there are exactly $3^{k+1}$ immediate successor-nodes. Let $\{t_i \such i<3^{k+1} \}$ enumerate all nodes $t\in p$ of length $n_k +1$. We construct $q\leq_k p$ together with a decreasing sequence $p=q_0\geq q_1\geq ... \geq q_{3^{k+1}}=q$. Assume we want to construct $q_{j+1}$. Find $p_j\in D $ so that $p_j \leq q_j \restric t_j$ (this is always possible since $D$ is dense). We define $q_{j+1}$ to be the condition which is obtained from $q_j$, by copying $p_j$ above each node in $q_j$ of length $n_k + 1$. 
 More precisely:
\begin{align*}
q_{j+1} := \{t\in q_j \such  (|t|\leq n_{k}+1 \vee & (|t|> n_{k}+1 \wedge \exists s \in p_j \; \forall n\in \omega  \\
& (n_{k}< n < |t| \rightarrow s(n)=t(n)) )) \}.
\end{align*}

It follows from the construction that for $q:= q_{3^{k+1}}$ and $j<3^{k+1}$ we must have $q\restric t_j \leq p_j$. In particular, we have that $q \leq _k p$. Put $E:= \{p_j \such j<3^{k+1}\}$. We want to check that $E$ is pre dense below $q$. Therefore, let $r\leq q$ be given. Then there is $j<3^{k+1}$ such that $r\restric t_j \leq q\restric t_j$. But also $q\restric t_j \leq p_j\in E$ and so $r$ and $p_j$ are compatible via $r\restric t_j$.  
\end{proof}

\begin{lemma} \label{quasi-pure}
$\mathias$ satisfies quasi pure decision, i.e., for every open dense $D \subseteq \mathias$, $p \in \mathias$, there is $q \leq_0 p$ satisfying what follows: if there exists $q' \leq q$ such that $q' \in D$, then $q \restric \stem(q') \in D$ as well.    
\end{lemma}

\begin{proof}
Let $p\in \mathbb{T}$ and $D\subseteq \mathbb{T}$ open dense be given. We construct a fusion sequence $p=q_0\geq_0 q_1 \geq_1 ... $ such that the fusion $q=\bigcap_k q_k$ witnesses quasi pure decision. Assume we are at step $k+1$ of the construction i.e. we have already constructed $q_{k}$. Let $A_{q_{k}}= \{ n_i \such i\in \omega \}$ be the corresponding set of splitting levels. Let $\{t_j \in q_{k} \such j\in 3^k \}$ enumerate all nodes in $q_k$ of length $n_k$. Similar to above we construct a decreasing sequence $q_{k} = q^0_{k} \geq q^1_{k} \geq ... \geq q^{3^k}_{k}$. Assume we are at step $j<3^k$. There are two cases:\\
\textit{Case 1}: There is no stronger condition $p' \leq q^j_{k}$ in $D$ with $\stem (p') = t_j$. Then do nothing and put $q^{j+1}_{k}:=q^{j}_{k}$.\\
\textit{Case 2}: Otherwise there is a $p' \leq q^j_{k}$ in $D$ with $\stem (p') = t_j$. As in the proof above we define
\begin{align*}
q^{j+1}_k := \{t\in q^j_k \such  (|t|\leq n_{k}+1 \vee & (|t|> n_{k}+1 \wedge \exists s \in p' \; \forall n\in \omega  \\
& (n_{k}< n < |t| \rightarrow s(n)=t(n)) )) \};
\end{align*}
specifically $q^{j+1}_{k} \restric t_j = p'$. Finally defining $q_{k+1} := q^{3^k}_k$, we get that the corresponding two sets of splitting levels $A_{q_{k}}$ and $A_{q_{k+1}}$ coincide on the first $k+1$ elements and therefore $q_{k+1} \leq_{k} q_{k}$. This completes the construction.\\
Before showing that the fusion $q := \bigcap_k q_k$ witnesses quasi pure decision we make the following observation: Since in the $(k+1)$-th step in the construction of the fusion the $k$-th splitting level is fixed, we know for each $k \in \omega $ and $l> k$ that $q\leq_k q_l $. Therefore the two sets of splitting levels $ A_q$  and $A_{q_l}$ coincide on the first $l$ elements.\\ 
Now let $q' \leq q$ in $D$ be given. Put $t:=\stem (q')$. Again we denote the splitting levels of $q$ by $A_q = \{ n_k \such k\in \omega\}$ and take $n_{k}$ such that $|t| = n_{k}$. We look at the construction of $q_{k+1}$. Then there is $j<3^k$ with $t_j =t$. Since $q' \leq q \leq q^j_k$ and $q'\in D$ we know that in the construction of $q^{j+1} _k$ case 2 was applied i.e. $q^{j+1}_k\restric t =p'$ for some $p'\in D$. Thus, using openness of $D$ and $ q\restric t \leq q^{j+1}_k\restric t$, we also get $q\restric t\in D$.
\end{proof}

\begin{lemma}\label{t adds cohen reals} \
\begin{enumerate}
\item $\mathias$ does not satisfy pure decision.
\item $\mathias$ adds Cohen reals.
\end{enumerate}
\end{lemma}
\begin{proof} 

(1). We have to find a condition $p\in \mathias$ and a sentence $\varphi$ such that no $q\leq_0 p$ decides $\varphi$. We prove something slightly stronger: Given any $p\in \mathias$ we can find a sentence $\varphi_p$ such that there is no $q\leq_0 p$ deciding $\varphi_p$.\\
So let $p\in \mathias$ and $q\leq_0 p$  be given (i.e. $q\leq p \wedge\stem(p)=\stem(q)$). Let $\dot{z} $ be the $\mathias$-name for the generic real. It is clear that $\Vdash_\mathias \exists^\infty n \;\dot{z}(n)=2$. We can define a name $\dot{\sigma}_z\in \omega^\omega \cap V^\mathias$ such that
$$\Vdash_\mathias \dot{\sigma}_z (k)= k\text{-th }2 \text{ occurring in }\dot{z}.$$ 
This means that in any generic extension $V[z]$ the evaluation of $\dot{\sigma}_z$ enumerates the set $\{k\in \omega \such z(k)=2 \}\in V[z]$. For $k\in \omega $ we define 
$$\varphi_k := `` \text{there are even many } 1\text{'s occuring in } \dot{z} \text{ between } \dot{\sigma}_z(k) \text{ and } \dot{\sigma}_z(k+1)".$$
Put $k:= |\{n < |\stem(q)| \such \stem(q)(n)=2 \}|$ and let $n_0^q < n_1^q$ denote the first two splitting levels of $q$. Take $q_0,q_1 \leq q$ such that
\begin{enumerate}
\item $\stem(q_0)(n^q_0)=0$ and $\stem(q_0)(n^q_1)=2$,
\item $\stem(q_1)(n^q_0)=1$ and $\stem(q_1)(n^q_1)=2$.
\end{enumerate}
Then there are at least $k+1$ many $2$'s occurring in $\stem(q_i)$, therefore $\varphi_k$ is decided by $q_i, i\in2$ and we get 
$$q_0\Vdash \varphi_k \Leftrightarrow q_1 \Vdash \lnot \varphi_k.$$
This proves that $q$ does not decide $\varphi_k$.

\vspace{5mm}

(2). We now show with a similar idea that $\mathias$ adds Cohen reals. Again let $\dot{z}$ be the $\mathias$-name for the generic real and let $\dot{\sigma}_z$ be as above. For every $k \in \omega$, 
\begin{itemize}
\item $c(k)=0$ iff $|\{ i \in \omega: \dot{\sigma}_z(k) \leq i < \dot{\sigma}_z(k+1)\land \dot{z}(i)=1  \}|$ is even
\item $c(k)=1$ iff $|\{ i \in \omega:\dot{\sigma}_z(k) \leq i < \dot{\sigma}_z(k+1) \land \dot{z}(i)=1  \}|$ is odd.
\end{itemize}
Then $\Vdash_\mathias c\in 2^\omega$. We want to show that $c$ is Cohen. So fix $p \in \mathias$, $\sigma \in 2^{<\omega}$ and let $c_p \subseteq c$ be the part of $c$ decided by $p$. 
 We aim to find $q \leq p$ such that $q \force {c_p}^\conc \sigma \subseteq c$. This is sufficient to show that $c$ is Cohen. 

Let $k =|c_p|$, i.e. $k$ is minimal such  that $c(k)$ is not decided by $p$.  Define $p=q_0 \geq q_1 \geq \dots \geq q_{|\sigma|}$ by recursion as follows.

Assume we have constructed $q_j,j<|\sigma|$.  Let $n^j_0<n^j_1$ be the first two splitting levels of $q_j$. For $i\in 2$ take $t_i \in q_j$ of length $n^j_1+1$ so that $t_i(n_0^j)=i$ and $t_i(n^j_1)=2$. Put $q_j^i := q_j\restric t_i$. Then we must have  
\begin{align}
|\{m\in \omega \such n^j_0 \leq m < n^j_1 \wedge \stem(q_j^i)(m)=1  \}|=_{\mod 2} \sigma(j) \label{gl1}
\end{align}
for exactly one $i\in 2$. Let $q_{j+1} = q_j^i$ such that (\ref{gl1}) holds.\\
Then by construction, for every $j < |\sigma|$, $q_{|\sigma|} \force c(|c_p|+j) = \sigma(j)$, i.e., $q_{|\sigma|} \force {c_p}^\conc \sigma \subseteq c$.

\end{proof}

Before moving to the issue concerning the ideals $\ideal{I}_\mathias$ and $\ideal{N}_\mathias$, we have to clarify the space that we are interesting in working with. To understand the point let us consider the standard Mathias forcing $\poset{R}$. If we work in the Cantor space $2^\omega$ literally, then we end up with a trivial example to show that $\ideal{N}_\poset{R} \neq \ideal{I}_\poset{R}$, namely the set of ``rational numbers'', i.e., the set $Q:= \{ x \in 2^\omega \such \exists n \forall m \geq n (x(m)=0) \}$. In a similar fashion one can check that the sets $N_n := \{ x\in 3^\omega \such  x(i) \neq 2\;\forall i\geq n \}$ are $\mathias$-nowhere dense, but the union $\bigcup_{n \in \omega} N_n$ is not. We leave the straightforward proof to the reader.

For the same argument we specified in Remark 2, indeed the space we really refer to when we work with the standard Mathias forcing is not literally $2^\omega$, but is the subspace obtained via the identification of $[\omega]^\omega$ and $2^\omega$, i.e., the set $\{ x \in 2^\omega: \exists^\infty n (x(n)=1) \}$. In such a space the counterexample disappears and indeed we get $\ideal{I}_{\poset{R}}= \ideal{N}_{\poset{R}}$. 
The main difference we want to make is that $\poset{T}$ behaves completely differently. In fact even when we take the ``proper'' space $H:= \{x \in 3^\omega: \exists^\infty n (x(n)=2) \}$ we cannot show that $\ideal{N}_\mathias = \ideal{I}_\mathias$, as the following result highlights (where the ideals are considered in the space $H$).

\begin{lemma}\label{t^0 T^0} $\ideal{N}_\mathias \neq \ideal{I}_\mathias$. 
\end{lemma}

\begin{proof}
Given $z \in H$ consider $\sigma_z \in \omega^\omega$ as in the proof of the previous Lemma and also remind $c_z \in 2^\omega$ be as follows:
\begin{itemize}
\item $c_z(k)=0$ iff $|\{ i \in \omega: \sigma_z(k) \leq i < \sigma_z(k+1)\land {z}(i)=1  \}|$ is even
\item $c_z(k)=1$ iff $|\{ i \in \omega:{\sigma}_z(k) \leq i < {\sigma}_z(k+1) \land {z}(i)=1  \}|$ is odd.
\end{itemize}
Then define 
\[
M_n:= \{  z \in H \such \forall k \geq n (c_z(k)=0)\}. 
\]
We claim each $M_n$ is $\mathias$-nowhere dense, but $\bigcup_{n \in \omega}M_n$ is not. In fact given $n\in \omega$ and $p\in\mathias$ we can lengthen the stem of $p$ to get a stronger condition $p'\leq p$ such that $\{k<|\stem (p')| \such p'(k)= 2 \}$ has size $>n$. Let $A_{p'}:= \{n_i \such i\in \omega\}$. Now we take $t\in \splitting_2(p')$ extending $\stem(p')^\conc 2$ i.e., $t(n_0)=2$ such that $t(n_1)\neq 2$ and the set of $\{k > |\stem(p')| \such t(k)=1 \}$ is odd. Then $q:= p'\restric t^\conc 2$ has no common branch with $M_n$. On the other hand there is always a branch $z\in [p]\cap H$ such that for all $k>\stem (p)$, $c_z(k)=0$.
\end{proof}

\section{$\Gamma(\poset{P}) \Rightarrow \Gamma(\cohen)$}

We now prove a rather general result, showing how the ``Cohen coding'' allows us to prove a classwise connection between $\poset{P}$-measurability and Baire property. Beyond its own interest, the technique used will also permit us to apply it in other specific cases that we will summarize along the paper, in particular to answer a question connected to the diagram of regularity properties at uncountable investigated in \cite{FKK16}. Recall that a family of sets $\Gamma$ is \emph{well-sorted} if it is closed under continuous pre-images and $\Gamma(\poset{P})$ stands for ``every set in $\Gamma$ is $\poset{P}$-measurable''.

\begin{proposition}\label{gamma p implies gamma c}
Let $\mathcal{X}$ be a set of size $\leq \kappa$ endowed with the discrete topology, $\mathcal{X}^\kappa$ the topological product space equipped with the bounded topology (i.e., the topology generated by $[t]:=\{ x \in \mathcal{X}^\kappa: x \supseteq t \}$ with $t \in \mathcal{X}^{<\kappa}$), $\mathbb{P}$ be a $<\kappa$-closed tree-forcing notion defined on $\mathcal{X}^{<\kappa}$. Assume there exist two maps $\varphi:\mathcal{X}^\kappa \rightarrow 2^\kappa$ and $\varphi^*: \mathcal{X}^{<\kappa} \rightarrow 2^{<\kappa}$ such that:
\begin{enumerate}
\item[a)] $\varphi$ is continuous,
\item[b)] $\forall i< \kappa \; \varphi(x)\restric i = \varphi^*(x\restric i)$,
\item[c)] $\forall q\in \mathbb{P}\; \forall s\in 2^{<\kappa }\; \exists \sigma \in q $ such that $\varphi^*(\sigma) \supseteq \varphi^*(\stem(q))^\conc s.$  
\end{enumerate}
Then $\Gamma(\mathbb{P} )$ implies $\Gamma(\mathbb{C})$.
\end{proposition}
We note that the second condition implies $\varphi[[p]]\subseteq [\varphi^*(\stem(p))]$ for each $p\in \mathbb{P}$. The third condition intuitively means that the map $\varphi^*$ is below any condition almost surjective. The key step for the proof is the following lemma.

\begin{lemma}
Let $\mathbb{P},\varphi,\varphi^*$ be as in the Proposition and $X\subseteq 2^{\kappa}$. Define $Y:= \varphi^{-1}[X]$. Assume there is $q\in \mathbb{P}$ such that $Y\cap[q]$ is $\mathbb{P}$-comeager in $[q]$. Then $X\cap [\varphi^*(\stem(q))]$ is comeager in  $[\varphi^*(\stem(q))]$.
\end{lemma}
\begin{proof}
We are assuming $Y \cap [q]$ is $\mathbb{P}$-comeager, for some $q \in \mathbb{P}$. This implies that there is a collection $\{A_\alpha \such \alpha<\kappa \wedge A_\alpha \text{ is } \mathbb{P}\text{-open dense in }[q]\}$ such that $\bigcap_\alpha A_\alpha \subseteq [q]\cap Y$. W.l.o.g. assume $A_\alpha \supseteq A_\beta$, whenever $\alpha<\beta<\kappa$. Let $t = \varphi^* (\stem(q))$. We want to show that $\varphi[Y]\cap t = X\cap t$ is comeager in $[t]$ i.e., we want to find $\{B_\alpha \such \alpha<\kappa \}$ open dense sets in $[t]$ such that $\bigcap_\alpha B_\alpha \subseteq X\cap [t]$. Given $\sigma\in \kappa^{<\kappa}$ we recursively define on the length of $\sigma$ a set $\{q_\sigma \such \sigma \in\kappa^{<\kappa}  \}\subseteq \mathbb{P}$ with the following properties:
\begin{description}
\item[1.] $q_{\langle\rangle} = q$,
\item[2.] $\forall \sigma \in \kappa^{<\kappa}\; \bigcup_{i} [\varphi^* (\stem(q_{\sigma^\conc i }))]$ is open dense in $[\varphi^* (\stem(q_{\sigma}))]$,
\item[3.] $\forall \sigma \in \kappa^{<\kappa} \forall i \in \kappa \; ([q_{\sigma^\conc i}]\subseteq A_{|\sigma|} \wedge q_{\sigma^\conc i} \leq q_{\sigma})$.
\end{description}
Assume we are at step $\alpha= |\sigma|$. Fix $\sigma\in\kappa^\alpha$ arbitrarily and then put $t_\sigma = \varphi^* (\stem (q_\sigma))$. We first make sure that $2.$ holds. Therefore let $\{ s_i \such i<\kappa \}$ enumerate $2^{<\kappa}$. By condition $c)$ from Proposition \ref{gamma p implies gamma c} we can find $p_i \leq q_\sigma$ such that $\varphi^* (\stem(p_i)) \supseteq {t_\sigma}^\conc s_i$. Since each $A_\alpha$ is $\mathbb{P}$-open dense in $[q]$ we can find for each $i<\kappa$ an extension $q_i \leq p_i$ such that $[q_i] \subseteq \bigcap_{\alpha \leq |\sigma|} A_\alpha$. This ensures that also $3.$ holds and we put $q_{\sigma^\conc i} := q_i$. 
At limit steps $\lambda$, we put for every $\sigma \in \kappa^\lambda$, $q_{\sigma}:= \bigcap_{\beta<|\sigma|} q_\beta$. 
Finally we put  $B_{\alpha}:= \bigcup\{ \varphi[[q_{\sigma}]]\such \sigma\in \kappa^\alpha \}$. We have to check that $\bigcap_\alpha B_\alpha \subseteq X\cap [t]$. Since $t= \stem(q)$ and $q_\sigma \leq q$ we get $\varphi[[q_\sigma]] \subseteq \varphi [[q]]\subseteq [\varphi^*(t)]$ and therefore $B_\alpha \subseteq [t]$ for each $\alpha \in \kappa$. On the other hand by construction of $B_{\alpha+1}$ we know $\varphi^{-1}[B_{\alpha+1}] \subseteq A_\alpha$ and hence $\varphi^{-1}[\bigcap_\alpha B_\alpha]\subseteq \bigcap_\alpha A_\alpha$ which implies $\bigcap_\alpha B_\alpha \subseteq X$. 
\end{proof}
\begin{proof}[Proof of the proposition] Let $X\in \Gamma $ be given and put $Y:= \varphi^{-1} [X]$. Then also $Y\in\Gamma$, since $\Gamma$ is well-sorted and $\varphi$ is continuous.  We now use the lemma to show that for every $t \in 2^{<\kappa}$ there exists $t' \supseteq t$ such that $X \cap [t']$ is meager or $X \cap [t']$ is comeager.  

Fix $t \in 2^{<\kappa}$ arbitrarily and pick $p \in \mathbb{P}$ such that $\varphi^*(\stem(p))\supseteq t$. By assumption $Y$ is $\mathbb{P}$-measurable, and so:
\begin{itemize}
\item in case there exists $q \leq p$ such that $Y \cap [q]$ is $\mathbb{P}$-comeager; put $t':=  \varphi^*(\stem(q))$. By the lemma above, $X \cap [t']$ is comeager in $ [t']$;
\item in case there exists $q \leq p$ such that $Y \cap [q]$ is $\mathbb{P}$-meager, then apply the lemma above to the complement of $Y$, in order to get $X \cap [t']$ be meager in $ [t']$, with $t':=  \varphi^*(\stem(q))$. 
\end{itemize}

By the remark directly after Definition \ref{meager} this suffices to complete the proof.

\end{proof}

\begin{proposition} \label{prop-mathias-baire}
Let $\Gamma$ be a well-sorted family of sets. Then 
\[
\Gamma(\mathias) \Rightarrow \Gamma(\cohen).
\]

\end{proposition}
\begin{proof}
Consider $H:= \{ x\in 3^\omega \such \exists^\infty n \; x(n)=2 \}$. As we remarked right above Lemma \ref{t^0 T^0}, $H$ is $\mathias$-comeager. Thus we have for each set $X\subseteq 3^\omega$:
\begin{align*}
X\text{ is }\mathbb{T}\text{-measurable }\Leftrightarrow X\cap H\text{ is }\mathbb{T}\text{-measurable.}
\end{align*}
Since we are only concerned with $\mathias$-measurability we can work with the set $H$ instead of the whole space $3^\omega$. We want to apply Proposition \ref{gamma p implies gamma c}. For an element $x\in H$ let $A_x=\{n_i \such i<\omega \}$ be an increasing enumeration of all $n\in \omega$ such that $x(n)=2$. This is by definition of $H$ an infinite set. Using this notation we define a function $\varphi :H \rightarrow 2^\omega$ via:
\begin{equation*}
   \varphi (x)(i)=
   \begin{cases}
     0 & \text{if } |\{j<\omega \such n_i < j < n_{i+1} \wedge x(j)=1 \} | \text{ is even } \\
     1 & \text{else.}
   \end{cases}
\end{equation*}
Note that $\varphi$ is surjective but not injective and observe that $\varphi$ induces a map $\varphi^*:3^{<\omega} \rightarrow 2^{<\omega}$ such that for each $x\in H$ and $i<\omega$ we have $\varphi(x)\restric i =\varphi^* (x\restric n_i)$. We have to check that $a), b) $ and $c)$ from Proposition \ref{gamma p implies gamma c} are satisfied. Condition $b)$ is clear. For condition $a)$  we have to show that the pre-image of a basic open set in $2^\omega$ is open in $H$ (regarding the induced topology of $3^\omega$ on $H$). Therefore let $s\in 2^{<\omega}$ be given. It follows 
\begin{align*}
\varphi^{-1} [[s]]=\bigcup_{t\in 3^{<\omega}, \varphi^*(t)=s } [t]\cap H
\end{align*}
which is a union of basic open sets in $H$. \\
So we are left to show that $c)$ holds as well. Therefore fix $q\in \mathbb{T}$ and $s\in 2^{<\omega}$. Let $A_q=\{n_i \such i<\omega\}$ be the corresponding set of splitting levels and $s=(i_1,\dots , i_k)$. Then we can lenghten $\stem(q)$ in order to have the parity of 1s between two subsequent 2 according to the corresponding $i_j$, that means we find $t \in q$ such that $\varphi^*(t) \supseteq \varphi^*(\stem(q))^\conc s$.

So we are able to apply Proposition \ref{gamma p implies gamma c} and get $\Gamma(\mathias) \Rightarrow \Gamma(\cohen)$.

\end{proof}

\section{Some results for the uncountable case}
In this section we investigate some issues concerning Table 2. We will always assume that $\kappa$ is an uncountable regular cardinal such that $\kappa = 2^{<\kappa}$.

\begin{definition}[Club $\kappa$-Miller forcing $\clubmiller$]
A tree $p\subseteq \kappa^{<\kappa}$ is called $\kappa$-Miller tree if it is pruned, $<\kappa$-closed and
\begin{itemize}
\item[(a)] for every $s\in p$ there is an extension $t\supseteq s$ in $p$ such that $\successor(t,p)\subseteq \kappa$ is club. Such a splitting node $t$ is called \emph{club-splitting}.
\item[(b)] for every $x\in [p]$ the set $\{ \alpha <\kappa \such x\restric \alpha \text{ is club-splitting } \}$ is club. 
\end{itemize}
\end{definition}
 \textit{Remark:} Both (a) and (b) ensure that $\clubmiller$ is a $<\kappa$-closed forcing. The set of trees that consist of nodes that are either club-splitting or not splitting is a dense subset of $\clubmiller$.

The following result highlights the connection with $\kappa$-Cohen reals. We remark that a similar result (though in a different context, dealing with a version of $\miller_\kappa$ satisfying (a) but not (b)) has been proven by Mildenberger and Shelah in \cite{MS**}.

\begin{proposition} \label{prop:miller-baire}
Let $\Gamma$ be a well-sorted family of subsets of $\kappa$-reals. Then $\Gamma(\miller^\club_\kappa) \Rightarrow \Gamma(\cohen)$.
\end{proposition}

\begin{proof}
We introduce a coding function $\varphi^* : \kappa^{<\kappa} \rightarrow 2^{<\kappa}$. Therefore  fix a $\kappa $ sized family $\{S_t\subseteq \kappa \such t\in 2^{<\kappa} \}$ of pairwise disjoint stationary sets such that the union of all $S_t$'s covers $ \kappa$  (this is possible since we assume $\kappa=2^{<\kappa}$). Let $\sigma \in \kappa^{<\kappa}$. We define $\varphi^* (\sigma) = {t_{i_{0}}}^\conc {t_{i_1}}^\conc \dots ^\conc {t_{i_{\alpha}}}^\conc \dots$, with $\sigma(\alpha)\in S_{t_{i_\alpha}}$ for all $\alpha<|\sigma|$. Then $\varphi^*$ induces a function $\varphi :\kappa^\kappa \rightarrow 2^\kappa$ via $\varphi (x) \restric\alpha:= \varphi^*(x\restric \alpha)$.


It is easy to see that such maps $\varphi$ and $\varphi^*$ satisfy the three conditions in Proposition \ref{gamma p implies gamma c}; $a)$ and $b)$ are clear, so we only check condition $c)$. So fix $q\in \miller^\club_\kappa$ and $t\in 2^{<\kappa}$. Let $\tau= \stem(q)$. Since $\successor(\tau,q)\subseteq \kappa$ is club and $S_t$ is stationary, we can pick $\beta\in S_t \cap \successor(\tau,q)$. Then $\tau^\conc \beta \in q$ and $\varphi^*(\tau^\conc \beta)=\varphi^*(\tau) ^ \conc t$. Using Proposition \ref{gamma p implies gamma c} we obtain $\Gamma(\miller^\club_\kappa) \Rightarrow \Gamma(\cohen)$ as desired. 
\end{proof}

\begin{remark} \label{miller-cohen}
The map $\varphi$ we used in Proposition \ref{prop:miller-baire} allows us to read off a Cohen $\kappa$-real from the $\clubmiller$-generic. Indeed, let $\{S_t\subseteq \kappa \such t\in 2^{<\kappa} \}$, $\varphi^*$ and $\varphi$ be as above. Let $\dot{z} $ be the $\clubmiller$-name for the generic $\kappa$-real and $\dot{c}$ the $\clubmiller$-name such that $\Vdash_{\clubmiller} \dot{c}= \varphi (\dot{z})\in 2^\kappa$. We claim that $\dot{c}$ is $\kappa$-Cohen in every generic extension. Therefore fix $p\in\clubmiller$ and let $c_p \in 2^{<\kappa}$ be the initial part of $\dot{c}$ decided by $p$ so $c_p = \varphi^*(\stem (p))$. Let $t\in 2^{<\kappa}$ be given. We want to find $q\leq p$ such that $q\Vdash {c_p}^\conc t \subseteq \dot{c}$. Since $\stem(p)$ is club-splitting we can find an $\alpha_0 \in S_t \cap \{\alpha<\kappa \such \stem (p)^\conc \alpha \in p  \}$ and take $q$ to be $p\restric \stem(p)^\conc \alpha_0$ i.e. $\stem(q)$ extends $\stem(p)^\conc \alpha_0$. This implies that $\varphi^*(\stem(q))\supseteq {c_p}^\conc t$ and therefore $q\Vdash {c_p}^\conc t \subseteq \dot{c}$ as demanded.

We also remark that the fact that $\clubmiller$ adds Cohen $\kappa$-reals is not new and it was proven in \cite{BBFM}, even if the authors use a different coding map. 
\end{remark}

Differently from $\mathias$, the Cohen-like behaviour of the $\clubmiller$-generic does not have an impact on the ideals, as shown in the next result.
\begin{lemma}
$\ideal{N}_{\miller^\club_\kappa} = \ideal{I}_{\miller^\club_\kappa}$
\end{lemma}

\begin{proof}
The proof is rather standard. We report a sketch of it here just for completeness. Given $\{D_i: i <\kappa  \}$ a family of $\miller^\club_\kappa$-open dense sets and $p \in \miller^\club_\kappa$ we simply construct a fusion sequence $\{q_i:i<\kappa  \}$ so that $q:= \bigcap_{i<\kappa}q_i \leq p$, for every $i<\kappa$, $[q_i] \subseteq D_i$, and for every $j<i $, $q_i \leq_j q_j$, i.e., $q_i \leq q_j$ and for every $j\leq i$, $\splitting_j(q_i)=\splitting_j(q_j)$. This can be done via an easy recursive construction: at limit steps $i$, simply put $q_i:= \bigcap_{j<i}q_j$; at successor step $i+1$, for every $t \in \splitting_{i}(q_i)$, pick $p(t) \leq q_i \restric t$ such that $p(t) \in D_i$, and then put $q_{i+1}:= \bigcup \{ p(t): t \in \splitting_i(q_i) \}$.
\end{proof}

\begin{definition}[$\kappa$-Mathias forcing $\mathbb{R}_\kappa$] A $\kappa$-Mathias condition is a tuple $(s,A)$, where $s\in [\kappa]^{<\kappa}, A \in [\kappa]^\kappa $ such that $\operatorname{sup} ( s ) < \operatorname{min} (A)$. The partial order on $\mathbb{R}_\kappa$ is defined by:
\begin{alignat*}{1} 
(s,A)\leq (t,B) \Leftrightarrow t\subseteq s,A\subseteq B\text{ and }t\setminus s \subseteq A.
\end{alignat*}
\end{definition}

\begin{lemma} \label{lemma:kappa-mathias-ideal}
$\ideal{N}_{\mathbb{R}_\kappa}\neq \ideal{I}_{\mathbb{R}_\kappa}$
\end{lemma}
\begin{proof}
We first clarify what is meant with $\ideal{N}_{\mathbb{R}_\kappa}$: $X\subseteq [\kappa]^\kappa$ is called $\mathbb{R}_\kappa$-nowhere dense if for each $(s,A)\in \mathbb{R}_\kappa$ there is a stronger condition $(t,B)\leq (s,A)$ such that
\begin{align}
 \forall x\in X \forall	y\in [B]^\kappa (x\neq t\cup y). \label{equa}
\end{align} 
We define an equivalence relation on the set of countably infinite subsets of $\kappa$. For $a,b\in [\kappa]^\omega$ let $a\sim b :\Leftrightarrow |a\triangle b|<\omega$. We fix a system of representatives. For $a\in [\kappa]^\omega$ we denote the representative of $\{ b\in [\kappa]^\omega \such b\sim a \}$ with $\tilde{a}$. Then we define a coloring function $C: [\kappa]^\omega \rightarrow \{0,1 \}$ as follows:
\begin{equation*}
   C(a)=
   \begin{cases}
     0 & \text{if } |a\triangle \tilde{a}| \text{ is even } \\
     1 & \text{else.}
   \end{cases}
\end{equation*}
We can identify $x\in [\kappa]^\kappa$ with it's increasing enumeration $\chi:\kappa\rightarrow\kappa$ given by $\chi(\xi):=\min\{ x\setminus \bigcup_{\alpha<\xi} \chi(\alpha)\}$. Let $\{ \alpha_i \such i<\kappa \}$ enumerate the limit ordinals $<\kappa$. For $x\in [\kappa]^{ \kappa}$ and $i<\kappa$ we define the countable set $b^x_i := \{ x(\xi) \such \alpha_i < \xi < \alpha_{i+1} \}\subseteq \kappa$.\\

\noindent\textit{Claim:} The set $X_i:= \{ x\in [\kappa]^{ \kappa} \such \forall j>i \; C (b^x_j)=0 \}$ is $\mathbb{R}_\kappa$-nowhere dense for all $i<\kappa$, but their union is not. \\\\
\textit{Proof of the claim}. Let $(s,A)$ be a $\kappa$-Mathias condition and $i<\kappa$ be given. Fix $j>i$. Then $A\subseteq \kappa$ is of size $\kappa$. By removing at most one element of $A$, we find $A'\subseteq A$ such that $C(b^{A'}_j)=1$. We extend $s$ with the first $\alpha_{j+1}$ elements of $A'$ to get $t:= s \cup \{A'(\xi) \such \xi\leq \alpha_{j+1}\} \in \kappa^{ <\kappa}$. Now we can shrink $A'$ to $B:= A'\setminus (A'(\alpha_{j+1} )+1) $ in order to obtain a $\kappa$-Mathias condition $(t,B)\leq (s,A)$ fulfilling the requirement (\ref{equa}). This proves the claim.

However the union $X:=\bigcup_{i<\kappa} X_i$ can not be $\mathbb{R}_\kappa$-nowhere dense. In fact, let $(s,A)$ be a $\kappa$-Mathias condition. We can always find for $i>\operatorname{otp}(s)$ a subset $B\subseteq A$ of size $\kappa$ such that $C(b^B_j)=0$, for all $j>i$ and hence (\ref{equa}) is false for $X_i$ and $(s,B)$. 
\end{proof}

(The coloring introduced above requires AC. However the result needs not AC, as we can also consider another kind of coloring, as noted by Wohofsky and Koelbing during the writing of \cite{KKLW**}: fix $S \in [\kappa]^\kappa$ stationary and co-stationary and define the coloring $C: [\kappa]^\omega \rightarrow \{ 0,1 \}$ by $C(a):=0$ iff $\sup a \in S$.)

\begin{remark} \label{reamrk-kappa-hechler}
Proposition \ref{gamma p implies gamma c} also applies for $\poset{P} \in \{ \poset{D}_\kappa, \poset{E}_\kappa \}$. The coding function $\varphi: \kappa^\kappa \rightarrow 2^\kappa$ we need in this case is given by $\varphi(x)(i)= x(i) \text{ mod } 2$, similarly to the $\omega$-case. It is straightforward to prove that such a $\varphi$ (and the natural corresponding $\varphi^*$) satisfies the required properties of Proposition \ref{gamma p implies gamma c}.
\end{remark}
 $st$

\end{document}


\endinput